%% file: main.tex
\documentclass[10pt,a4paper,oneside]{amsart}
\input{top}

\title{Cobordism, spin structures, and profinite completions}

\DeclareMathOperator{\bTor}{\mathbf{Tor}}
\DeclareMathOperator{\bExt}{\mathbf{Ext}}
\DeclareMathOperator{\bHom}{\mathbf{Hom}}
\DeclareMathOperator{\PMod}{\mathrm{PMod}}
\DeclareMathOperator{\DMod}{\mathrm{DMod}}
\DeclareMathOperator{\FMod}{\mathrm{FMod}}
\author{Sam Hughes}
\address{Sam Hughes, Max-Planck-Institut für Mathematik, Vivatsgasse 7, 53111 Bonn, Germany}
\email{sam.hughes.maths@gmail.com; hughess@mpim-bonn.mpg.de}
\urladdr{https://samhughesmaths.github.io}
\author{Andrew Ng}
\address{Andrew Ng, Mathematisches Institut, Universität Bonn, Endenicher Allee 60, 53115 Bonn, Germany}
\email{clan@math.uni-bonn.de}
\urladdr{https://clandrei.github.io/}

\date{\today}
\subjclass[2020]{}

\begin{document}
\begin{abstract}
    Let $M$ and $N$ be smooth closed connected aspherical manifolds with good (in the sense of Serre) fundamental groups $G$ and $H$.  If $\widehat G\cong \widehat H$, then $M$ and $N$ are cobordant and the signatures of $M$ and $N$ agree modulo $8$.  Moreover, $M$ is spin (resp. spin$^\mathbb{C}$) if and only if $N$ is spin (resp. spin$^\mathbb{C}$).   We consider some analogous results for compact connected aspherical manifolds with boundary.
\end{abstract}
\maketitle

\section{Introduction}\label{sec:intro}
The study of profinite properties of fundamental groups of manifolds with geometric structures has received a wealth of attention.  
In the case of $3$-manifolds, Xu proved at most finitely many $3$-manifold groups may have the same isomorphism class of profinite completion \cite{Xu2025rigid}; Xu's work built on deep results of Liu \cite{Liu2023,Liu2023duke}, Wilton--Zalesskii \cite{WZ,WZ2019JSJ}, Wilkes \cite{Wilkes2017Seifert,Wilkes2018JSJ,Wilkes2018graph,Wilkes2019JSJS2} and others \cite{FriedlVidussi2011,BridsonReidWilon2017,BridsonReid2018,BoileauFriedl2020,JaikinZapirain2020}. 
That not all $3$-manifolds are determined by the set of finite quotients of their fundamental groups can be seen in the work of Funar \cite{Funar2013}, Hempel \cite{Hempel2014}, and Nery \cite{Nery2020} (as well as the previously cited work of Wilkes). 
Breakthrough work of Bridson--McReynolds--Reid--Spitler showed that some hyperbolic $3$-manifold groups \cite{BridsonMcReynoldsReidSpitler2020} and Fuchsian triangle groups \cite{BridsonMcReynoldsReidSpitler2021} are determined up to isomorphism by their sets of finite quotients amongst all finitely generated residually finite groups. 

In dimension $4$ some work towards recognising geometric structures has been done \cite{MaWang2022};  distinguishing higher dimensional Seifert fibre spaces has also been studied \cite{Piwek2023,MaWang2023}.  
Profinite recognition of flat manifolds and crystallographic groups has also been deeply investigated \cite{FinkenNeubuserPleskin1980,GrunewaldZalesskii2011,PiwekPopovicWilkes2021,Nery2021,Nery2024,CHMV2025,PaoliniSklinos2024,CarolilloPaolini2025}.  
For more information, the reader is referred to Reid's ICM notes \cite{Reid2018ICM} and R\'emy's S\'eminaire Bourbaki \cite{Remy2024}.

For a paracompact space $X$ and a vector bundle $E\to X$, its \emph{Stiefel--Whitney classes} $w_i(E)$ are certain characteristic classes lying in $H^\ast(X;\FF_2)$.  In the case of the tangent bundle $TM\to M$ of a closed smooth manifold $M$, the Stiefel--Whitney classes $w_i(M)\coloneqq w_i(TM)$ obstruct various geometric properties of the manifold such as orientability, pin, and spin structures. After work of Thom \cite{Thom54}, they are also closely related to the problem of determining when two manifolds are (unoriented) cobordant. One can also define \emph{integral Stiefel--Whitney classes}, and in the case of a closed smooth manifold these are related to spin$^\CC$-structures.

We remark that every $3$-manifold admits a spin structure and every compact $4$-manifold admits a spin$^\CC$ structure \cite{HirzebruchHopf1958}.  It is known that the existence of spin structures on a closed $4$-manifold $M$ is closely related to its signature $\sigma(M)$ \cite{Rohlin29152,Teichner1993}.  Recall that the \emph{signature} $\sigma(M)\in\Z$ of a closed manifold $M$ is given by the signature of the intersection form on (the free part of) $H^{2k}(M;\Z)$ if $\dim M=4k$ and is defined to be $0$ otherwise. It was recently proved that every almost flat spin$^\CC$ manifold bounds a compact orientable manifold \cite{HanHuangZhang}.


In \cite{Sullivan1979}, Sullivan proved (using \cite{DeligneSullivan1975}) that every finite volume hyperbolic $n$-manifold has a finite cover which is stably parallelisable and hence admits a spin structure.  An alternative proof was given by Long--Reid in \cite{LongReid2020} and for all $n\geq 5$ the authors constructed orientable finite volume hyperbolic $n$-manifolds that do not admit a spin structure. In \cite{MartelliRioloSlavich2020}, Martelli--Riolo--Slavich constructed for all $n\geq 4$ orientable closed hyperbolic $n$-manifolds that do not admit a spin structure.  Riolo and Rizzi \cite{RioloRizzi2025}, for all $n\geq 4$ constructed a cusped orientable arithmetic hyperbolic $n$-manifold that does not admit a spin structure.  Chen \cite{Chen2025} proved, for all $n\geq 5$, the existence of infinitely many commensurability classes of closed orientable hyperbolic $n$-manifolds that do not admit any spin$^\CC$ structure.  There are also a number of papers constructing non-trivial Stiefel--Whitney classes on hyperbolic manifolds \cite{ReidSell2023,chen2025closedhyperbolic5manifolds,chen2025noncobordanthyperbolicmanifolds}.

Let $\widehat G$ denote the profinite completion of a group $G$. We say a group $G$ is \emph{good} (in the sense of Serre) if the natural map $\iota\colon G\to \widehat G$ induces isomorphism $\iota^n\colon \mathbf{H}^n(\widehat G;M)\to H^n(G;M)$ for every finite coefficient module $M$.  Here $\mathbf{H}^\ast(-,-)$ denotes continuous cohomology.

\begin{thmx}\label{thmx:A}
    Let $M$ and $N$ be smooth closed connected aspherical manifolds with good fundamental groups $G$ and $H$.  If $\widehat G\cong \widehat H$, then
    \begin{enumerate}
        \item $M$ and $N$ are (unoriented) cobordant;
        \item $M$ admits a spin structure if and only $N$ admits a spin structure;
        \item $M$ admits a spin$^\CC$ structure if and only $N$ admits a spin$^\CC$ structure.
        \item $\sigma(M)\equiv \sigma(N)\pmod{8}.$
    \end{enumerate}
\end{thmx}

The class of good groups is vast: including all virtually compact special groups  groups (in the sense of \cite{HW08}) \cite[Proposition 3.2]{WZ} and hence all cubulated negatively curved manifolds (after Agol \cite{Agol}).  This latter class includes all closed hyperbolic 3-manifold groups and all arithmetic lattices of simplest type in hyperbolic $n$-space \cite{BHW} (this includes examples from Chen's previously cited work \cite{Chen2025}). Further examples include all virtually polycyclic groups \cite[Theorem on pg~1708]{Lorensen2008}.  The property behaves well with respect to extensions \cite[Theorem~2.5]{Lorensen2008}.

\begin{remark} \label{rmk:csp}
    Higher rank lattices with the congruence subgroup property in semisimple Lie groups are never good \cite[Proposition 5.1]{GrunewaldJaikinZalesskii2008}. Note that without goodness one can construct examples of torsion-free higher rank uniform lattices $\Lambda_1\leqslant G_1$ and $\Lambda_2\leqslant G_2$ with isomorphic profinite completions such that the resulting closed locally symmetric manifolds are not even of the same dimension, see \Cref{ex:diffDims}.
\end{remark}

A key step in \Cref{thmx:A} is in relating the Stiefel--Whitney classes of the manifolds $M$ and $N$.  The precise relation is provided by the next theorem. Recall that a \emph{Poincar\'e complex} is a finite CW complex satisfying Poincar\'e duality.

\begin{duplicate}[\Cref{thmx:Invariance_of_SW_classes}]
    Let $M$ and $N$ be connected aspherical Poincar\'e complexes with good fundamental groups $G$ and $H$, respectively.  If $\Theta\colon \widehat G\to\widehat H$ is an isomorphism, then there exists an $\FF_2$-algebra isomorphism $\theta\colon H^\ast(N;\FF_2)\to H^\ast(M;\FF_2)$ such that $\theta(w_i(N))=w_i(M)$ and $\theta(w(N))=w(M)$.
\end{duplicate}

\subsection*{Structure of the paper}
In \Cref{sec:prelims} we recount the necessary background from algebraic topology, (profinite) group cohomology, and quadratic form theory we will need.  

In \Cref{sec:Steenrod} we recount Epstein's theory of Steenrod squares and prime power operations \cite{Epstein} and explain how they apply to both group cohomology and profinite group cohomology.  

In \Cref{sec:detect} we prove our main results.  More precisely, in \Cref{sec:detect:SW} we prove \Cref{thmx:Invariance_of_SW_classes} and consider analogous statements for manifolds with boundary.  In \Cref{sec:detect:spin} we prove profinite invariance of the existence of spin structures (\Cref{spin}).  In \Cref{sec:detect:SWnums} we prove profinite invariance of the Stiefel--Whitney numbers (\Cref{equalSWnumbers}) and deduce our result about (unoriented) cobordisms \Cref{cobordant}.  In \Cref{sec:detect:Bocksteins} we show that certain Bockstein maps can be detected in profinite group cohomology.  In \Cref{sec:detect:intSW} we show that non-vanishing of integral Stiefel--Whitney classes is a profinite invariant (\Cref{detectSW_Z}) and prove profinite invariance of the existence of spin$^\CC$ structures (\Cref{spinC}).  In \Cref{sec:detect:Pmod3} we show that the modulo $3$ reductions of the integral Pontryagin classes are profinite invariants (\Cref{thm:PontrayginClassesMod3}) and deduce that the Pontryagin numbers modulo $3$ are profinite invariants (\Cref{thm:PontrayginMod3}).  In \Cref{sec:detect:sigma} we prove that the signature of the intersection form modulo $8$ is a profinite invariant (\Cref{intersectionForm}).  In \Cref{sec:detect:thmA} we prove \Cref{thmx:A}.  

Finally, in \Cref{sec:example} we detail an example of a pair of aspherical manifolds $M$ and $N$ with $\pi_1 M$ and $\pi_1N$ residually finite and satisfying $\widehat{\pi_1 M}\cong \widehat{\pi_1 N}$ but $\dim M=128$ and $\dim N=112$.

\subsection*{Acknowledgements}
Sam Hughes was supported by a Humboldt Research Fellowship at Universit\"at Bonn. Andrew Ng was supported by funding from the European Union (ERC, SATURN, 101076148).  Both authors were supported by the Deutsche Forschungsgemeinschaft (EXC-2047/1 - 390685813). The authors thank Wolfgang L\"uck for helpful conversations and suggesting we consider the signature.  The authors thank Holger Kammeyer for providing us with \Cref{ex:diffDims}. The authors thank the anonymous referee for carefully reading an earlier version and making a number of helpful suggestions. This paper forms part of the second author's PhD thesis.

\section{Preliminaries}\label{sec:prelims}

\subsection{Notation} \label{sec:prelims:notation}
In different sections of the paper we will work in different module categories. At the start of each section we will say which module categories are relevant for that section. We also use the following conventions:
\begin{itemize}
    \item An algebra/ring/module $M$ is said to be finite if $|M| < \infty$. 
    \item We use $\mathbf{H}_*$ and $\mathbf{H}^*$ to denote continuous homology and cohomology respectively. 
    \item For a finitely generated module $M$ over a commutative PID $R$, we denote by $M^{\tf}$ the $R$-submodule that is the torsion-free part of $M$. 
    \item For a (rational) prime $p$ we denote by $\Z_p$ the $p$-adic integers and $\FF_p$ the field with $p$ elements.
\end{itemize}

\subsection{Background on characteristic classes}\label{sec:prelims:characteristic}

We first define Stiefel--Whitney classes.

\begin{defn}[Stiefel--Whitney classes of a space]
    Let $X$ be a paracompact topological space and let $E\to X$ be a vector bundle.  The \emph{total Stiefel--Whitney class} $w(E)=\sum_{k\geq0}w_k(E)\in H^\ast(X;\FF_2)$ is the unique class satisfying:
    \begin{enumerate}
        \item The class $w(L)$ of the tautological line bundle $L\to \RP^1$ is nontrivial.
        \item $w_0(E)=1$ and $w_k(E)=0$ for $k>\rank E$.
        \item $w(E\oplus F)=w(E)\smile w(F)$.
        \item For any map $f\colon Y\to X$ we have $w(f^\ast E)=f^\ast (w(E))\in H^\ast(Y;\FF_2)$.
    \end{enumerate}
    We call the class $w_k(E)\in H^k(X;\FF_2)$ the $k$th \emph{Stiefel--Whitney class}.
\end{defn}

\begin{defn}[Stiefel--Whitney classes of a manifold]
For a compact connected smooth $n$-manifold $M$ we define the \emph{Stiefel--Whitney classes} $w_k(M)$ of $M$ to be the Stiefel--Whitney classes of its tangent bundle $TM\to M$.
\end{defn}

Let $M$ be a Poincar\'e complex and let
\[\langle.,.\rangle : H^{n-k}(M) \times H_k(M) \rightarrow \FF_2 \]
be the pairing of cohomology and homology. \( M \) satisfies Poincaré duality for \( \FF_2\) (co)homology, hence we have
\[\mathrm{Hom}(H^{n-k}(M), \FF_2) \cong H_{n-k}(M) \cong H^k(M),\]
where a cohomology class \( y \in H^k(M) \) corresponds to the homomorphism \( x \mapsto \langle y \cup x, [M] \rangle \). In particular, one element of \( \mathrm{Hom}(H^{n-k}(M), \FF_2)\) is uniquely determined by the expression \( x \mapsto \langle \Sq^k(x), [M] \rangle \).  Note here that $\Sq^k\colon H^{n-k}(M;\FF_2)\to H^n(M;\FF_2)$ is the $k$th Steenrod square; we postpone the definition until \Cref{sec:Steenrod}. 
\begin{defn}
 Under the above isomorphism, the unique class corresponds to a degree \( k \) cohomology class \( v_k \). We call $v_k$ the \emph{$k$th Wu class of $M$} and $1+\sum_{i=1}^n v_k$ the \emph{total Wu class}.
\end{defn}

Note that we have the following:
\begin{align*}
   & \langle v_k \cup x, [M] \rangle = \langle \text{Sq}^k(x), [M] \rangle, \text{for all } x \in H^{n-k}(M), \text{and} \\
   & \langle v \cup x, [M] \rangle = \langle \text{Sq}(x), [M] \rangle, \text{for all } x \in H^*(M)
\end{align*}

For a closed connected smooth manifold, the Wu classes are related to the Stiefel--Whitney classes by the following theorem of Wu.

\begin{thm}[Wu]~\cite{Wu1950}\label{WuTheorem}
    Let $M$ be a closed connected smooth $n$-manifold. Then $w_k=\sum_{i=0}^k\Sq^{k-i}(v_i)$.
\end{thm}

The definition of the Wu classes only requires the existence of the Poincar\'e duality isomorphisms.  Thus, we may instead take this as the definition of the Stiefel--Whitney classes for a general Poincar\'e complex.

\begin{defn}[Stiefel--Whitney classes of a Poincar\'e complex]
For a compact connected Poincar\'e complex $M$ of dimension $n$ we define the \emph{Stiefel--Whitney classes} $w_k(M)$ of $M$ by the formula $w_k=\sum_{i=0}^k\Sq^{k-i}(v_i)$.
\end{defn}

\subsubsection{Stiefel--Whitney numbers}

\begin{defn}[Stiefel--Whitney numbers]
    Let $M$ be a compact connected Poincar\'e complex of dimension $n$. Given positive numbers $i_1, \cdots, i_k$ such that $\sum_{j=1}^k i_j=n$, this gives rise to a class $\prod_{j=1}^k w_{i_j}(M) \in H^n(M;\FF_2)$. Pairing this with the fundamental class of $M$ gives a number in $\FF_2$ known as a \emph{Stiefel--Whitney number} of $M$.
\end{defn}

It is a celebrated result of Thom that when $M$ is a closed connected smooth manifolds, $M$ is the boundary of a compact smooth manifold if and only if all of its Stiefel--Whitney numbers vanish \cite{Thom54}.  Equivalently, two $n$-manifolds $M$ and $N$ are cobordant if and only if they have the same Stiefel--Whitney numbers.

\subsubsection{Integral Stiefel--Whitney classes}

\begin{defn}[Integral Stiefel--Whitney classes]
    Let $\beta\colon H^k(M;\FF_2)\to H^{k+1}(M;\Z)$ denote the Bockstein homomorphism induced by the short exact sequence
    \[0\to \Z \xrightarrow{\times 2} \Z\to \Z/2 \to 0. \]
    We define $W_{k+1}=\beta(w_k)$ to be the \emph{$(k+1)$th integral Stiefel--Whitney class.}
\end{defn}

\begin{remark}
    For a compact manifold $M$, Stiefel--Whitney classes give rise to the following obstructions:
    \begin{enumerate}
        \item $w_1=0$ if and only if $M$ is orientable;
        \item more generally, the first $2^k-1$ Stiefel--Whitney classes vanish if and only if $M$ is $k$-orientable \cite{Hoekzema2020};
        \item $w_1=w_2=0$ if and only if $M$ admits a spin structure \cite{Haefliger1956};
        \item $W_3=0$ if and only if $M$ admits a pin$^\CC$ structure (see \cite[\S2]{Chen2025});
        \item $w_1=W_3=0$ if and only if $M$ admits a spin$^\CC$-structure \cite[Corollary D.4]{LawsonMichelsohn1989}.
    \end{enumerate}
\end{remark}

\subsubsection{Relative Stiefel--Whitney classes}
Kervaire \cite{Kervaire1957} defined relative characteristic classes with a view to extending certain results on closed manifolds to compact manifolds with boundary. We recall the parts of his work which are relevant for us.

Let $K$ be a finite CW-complex and $L$ a (non-empty) subcomplex of $K$. $(K,L)$ is a Poincar\'e duality pair of dimension $n$ over $\FF_2$ if, for any $0 \leq q \leq n$, $H^n(K, L ; \FF_2) = \FF_2$ and the cup product pairing
\[\cup: H^{n-q}(K, L ; \FF_2) \times H^q(K; \FF_2) \to H^n(K, L ; \FF_2) \]
is a perfect pairing. This definition naturally generalises to other coefficient rings. The Wu classes $(v_{K,L})_k \in H^k(K; \FF_2)$ are defined analogously by the requirement that, for any relative class $x \in H^{n-k}(K, L ; \FF_2)$ 
\[\Sq^k(x)=(v_{K,L})_k \cup x.\]
By \cite[Lemma 6.1]{Kervaire1957}, doubling $K$ along $L$ results in a Poincar\'e complex $M$. This has Wu classes $(v_M)_k$. Let $\iota: K \to M$ be the inclusion map. Kervaire then proves that \cite[Lemma 7.3]{Kervaire1957} \[ \iota^*((v_M)_k)= (v_{K,L})_k. \]
We define the total Wu class for the pair as in the absolute case.
When $K$ is a smooth manifold and $L$ is its boundary, $(K,L)$ is a Poincar\'e duality pair. Furthermore, we have that
\[w(K)=\iota^*(w(M))= \iota^*(\Sq (v_M))= \Sq(\iota^*v_M)= \Sq(v_{K,L}) \]
Hence, for a general Poincar\'e duality pair $(K,L)$, we may define its Wu classes as above, and its Stiefel--Whitney classes by the property \[w(K):=\Sq(v_{K,L}). \]
This is compatible with the original definition in the case that $K$ is a smooth manifold with boundary $L$.

\subsubsection{Chern and Pontryagin classes}\label{sec:chern_pontryagin}
In this section we define Chern classes, Pontryagin classes, and Pontryagin numbers.

\begin{defn}[Chern classes of a space]
    Let $X$ be a paracompact space and $E\to X$ be a complex vector bundle.  The \emph{total Chern class} $c(X)=\sum_{k\geq 0}c_k(X)\in H^\ast(X;\Z)$ is the unique class satisfying:
    \begin{enumerate}
        \item The class $c(L)$ of the tautological line bundle $L\to \CP^k$ equals $1-H$ where $H$ is Poincar\'e dual to the hyperplane $\CP^{k-1}\subset \CP^k$.
        \item $c_0(E)=1$ for all $E$.
        \item  $c(E\oplus F)=c(E)\smile c(F)$.
        \item For any map $f: Y\to X$ we have have $c(f^\ast E)=f^\ast(c(E))\in H^\ast(Y;\Z)$.
    \end{enumerate}
    We call the class $c_k(E)\in H^{2k}(X;\Z)$ the $k$th \emph{Chern class}.
\end{defn}

\begin{defn}[Pontryagin classes of a space]
    Let $X$ be a paracompact space and $E\to X$ a real vector bundle.  The $k$th Pontryagin class of $E$ is defined to be $p_k(E)=(-1)^kc_{2k}(E\otimes \CC)\in H^{4k}(X;\Z)$.
\end{defn}

\begin{defn}[Pontryagin classes of a manifold]
    For a compact connected smooth $n$-manifold $M$ we define the \emph{Pontryagin classes} $p_k(M)\in H^{4k}(M;\Z)$ of $M$ to be the Pontryagin classes of its tangent bundle $TM\to M$.
\end{defn}

In general the Pontryagin classes are not homotopy invariants \cite[Corollary~9.3]{Milnor1964} (see also \cite{Novikov1964,Tamura1957}).  However, for a smooth closed connected $n$-manifold $M$ Hirzebruch showed that the images $\overline p_k$ of the Pontryagin class $p_k$ under the modulo $3$ reduction map $H^{4k}(M;\Z)\to H^{4k}(M;\FF_3)$ are homotopy invariants \cite[pg~215]{Hirzebruch1954}.  To show this, Hirzebruch introduces modulo $3$ analogues of the Wu classes $s_3^i$ which are determined by the Steenrod $3$rd power operations $P^i$ in the sense that $P_i(u)=s_3^i\smile u\in H^n(M;\FF_3)$ for all $u\in H^{n-2i(p-1)}(M;\FF_3)$.  See \Cref{sec:Steenrod} for a definition of the operations $P^i$.  Hirzebruch then shows that $\overline p_k$ is a polynomial combination of the classes $s_3^i$.  Since this result is entirely algebraic one can define the modulo $3$ Pontryagin classes of any connected Poincar\'e duality space analogously to the Stiefel--Whitney classes.

\subsection{Profinite cohomology theory}\label{sec:prelims:profinite}
In this section we describe the cohomology theory of profinite groups we will use, which is defined analogously to the cohomology theory of discrete groups with all functors replaced by their continuous versions. We refer the reader to \cite{Brown1982} for background on cohomology of discrete groups. 

Let $R$ be a profinite ring (with identity). 
\begin{defn}
    A topological left $R$-module is a topological abelian group $M$ with a continuous function $R \times M \to M, (r,m) \mapsto r \cdot m$ such that 
    \[r \cdot (m_1+m_2) = r \cdot m_1 + r \cdot m_2 \text{ and } r_1 \cdot( r_2 \cdot m) = (r_1 \cdot r_2) \cdot m\]

    A topological $R$-module is \emph{profinite} if $M$ is profinite as an abelian group, and \emph{discrete torsion} if it has the discrete topology and all elements have finite order.
\end{defn}

Following \cite{Wilkes2024}, we denote by $\PMod(R)$ the category of profinite $R$-modules, $\DMod(R)$ the category of discrete torsion $R$-modules, and $\FMod(R)=\PMod(R)\cap \DMod(R)$ the category of (discrete) finite $R$-modules. It is an important fact that both $\PMod (R)$ and $\DMod (R)$ are abelian categories \cite[Theorem 6.1.2]{Wilkes2024}, and furthermore $\PMod (R)$ has enough projectives \cite[Corollary 6.1.21]{Wilkes2024}. However, it does not in general have enough injectives. In certain circumstances we can still build suitable resolutions and apply the continuous versions of the $\Tor$- and $\Ext$-functors, which we denote by $\bTor$ and $\bExt$, respectively, to compute continuous profinite cohomology. 
\begin{defn}
    Let $R$ be a profinite ring.  Let $Z$ be a profinite $R$-module and let $(F_n)$ be a projective resolution of $Z$ over $R$. Let $M$ be a discrete torsion $R$-module.  We define \[\bExt_R^k(Z, M)=H_k(\bHom_R(F_\bullet, M)),\]
    that is, the right derived functors of the continuous $R$-homomorphisms functor $\bHom_R(Z,-)$.  Similarly, we define $\bTor_k^R( -, M)$ as the (left) derived functors of $ - \widehat{\otimes} M$.  We refer the reader to \cite[Section 6.1]{RZ10} for details.
%
 %
\end{defn}

Let $\widehat G$ be a profinite group. We denote by $\widehat{\Z}\llbracket \widehat G \rrbracket$ the completed group algebra of $\widehat G$ \cite[Section 6.3]{Wilkes2024}.  


\begin{defn}\label{profinite homology}
Let $\widehat G$ be a profinite group.
\begin{enumerate}[leftmargin=*]
\item For $M\in \PMod(\widehat{\Z}\llbracket \widehat G \rrbracket)=\PMod(\widehat{\Z}\llbracket \widehat G \rrbracket)$, the \emph{profinite homology of $\widehat G$ with coefficients in $M$} is defined as 
\[\mathbf{H}_k(\widehat G;M)={\bTor}_k^{\widehat G}(\widehat{\Z},M) \in \PMod(\widehat{\Z}) .\]

\item For $A\in \DMod(\widehat{\Z}\llbracket \widehat G \rrbracket)$, the \emph{profinite cohomology of $\widehat G$ with coefficients in $A$} is defined as 
\[\mathbf{H}^k(\widehat G;A)=\bExt^k_{\widehat G}(\widehat{\Z},A) \in \DMod(\widehat{\Z}).\]

\end{enumerate}

\end{defn}

\begin{defn}
\begin{enumerate}
\item We say that $(G,\mathcal{H})$ is a {\em discrete group pair} if $G$ is a discrete group, and $\mathcal{H}=\{H_i\}_{i=1}^n$ is an ordered, non-empty, finite collection of subgroups in $G$.
\item We say that $(\widehat G,\mathcal{S})$ is a {\em profinite group pair} if $\widehat G$ is a profinite group, and $\mathcal{S}=\{S_i\}_{i=1}^n$ is an ordered, non-empty, finite collection of closed subgroups in $\widehat G$.
\end{enumerate}
\end{defn}

We now define relative profinite (co)homology following \cite[Definition 2.1]{Wilkes2019}. 

\begin{defn}\label{relative profinite homology}
Given a profinite group pair $(\widehat G,\mathcal{S})$, let $\mathbf{\Delta}_{\widehat G,\mathcal{S}}$ be the kernel of the augmentation homomorphism
$$ \mathop{\oplus}_{i=1}^{n}\widehat{\Z} \llbracket \widehat G/S_i \rrbracket \longrightarrow \widehat{\Z}.$$
Then $\mathbf{\Delta}_{\widehat G,\mathcal{S}}\in \PMod(\widehat{\Z}\llbracket \widehat G \rrbracket)$.
\begin{enumerate}[leftmargin=*]
\item For $M\in \PMod(\widehat{\Z}\llbracket \widehat G \rrbracket)$, the {\em profinite homology of $\widehat G$ relative to $\mathcal{S}$ with coefficients in $M$} is defined as 
\[\mathbf{H}_{k+1}(\widehat G,\mathcal{S};M)=\bTor_{k}^{\widehat G}(\mathbf{\Delta}_{\widehat G,\mathcal{S}}^\perp,M)\in \PMod(\widehat{\Z}).\]

\item For $A\in \DMod(\widehat{\Z}\llbracket \widehat G \rrbracket)$, the {\em profinite cohomology of $\widehat G$ relative to $\mathcal{S}$ with coefficients in $A$} is defined as 
 \[\mathbf{H}^{k+1}(\widehat G,\mathcal{S};A)=\bExt^k_{\widehat G}(\mathbf{\Delta}_{\widehat G,\mathcal{S}},A)\in \DMod(\widehat{\Z}).
\]

\end{enumerate}
\end{defn}
The discrete and profinite settings are related by the following:

\begin{prop}\cite[Proposition 5.11]{Xu2025}
Let $(G,\mathcal{H}=\{H_i\}_{i=1}^{n})$ be a discrete group pair. Let $\overline{H_i}$ be the closure of the image of $H_i$ in $\widehat{G}$ through the canonical homomorphism $\iota:G\to \widehat{G}$, and let $\overline{\mathcal{H}}=\{\overline{H_i}\}_{i=1}^{n}$. Then $(\widehat{G},\overline{\mathcal{H}})$ is a profinite group pair. 
\begin{enumerate}
\item For $M\in\PMod(\widehat{\Z}\llbracket \widehat G \rrbracket)$, $M$ can be viewed as a $\widehat{\Z}[G]$-module, and there is a natural homomorphism $H_\ast(G,\mathcal{H};M)\to \mathbf{H}_\ast(\widehat{G},\overline{\mathcal{H}};M)$ of $\widehat{\Z}$-modules induced by $\iota$ such that the following diagram commutes. 
\begin{equation*}
\begin{tikzcd}
\cdots \arrow[r] & {H_{k+1}(G,\mathcal{H};M)} \arrow[d, "\iota_\ast"] \arrow[r, "\partial"] & \mathop{\oplus}\limits_{i=1}^{n} H_k(H_i;M) \arrow[r] \arrow[d, "(\iota|_{H_i})_\ast"] & H_k(G;M) \arrow[r] \arrow[d, "\iota_\ast"] & {H_{k}(G,\mathcal{H};M)} \arrow[r, "\partial"] \arrow[d, "\iota_\ast"] & \cdots \\
\cdots \arrow[r] & {\mathbf{H}_{k+1}(\widehat{G},\overline{\mathcal{H}};M)} \arrow[r, "\partial"]   & \mathop{\oplus}\limits_{i=1}^{n} \mathbf{H}_k(\overline{H_i};M) \arrow[r]                      & \mathbf{H}_k(\widehat{G};M) \arrow[r]              & {\mathbf{H}_{k}(\widehat{G},\overline{\mathcal{H}};M)} \arrow[r, "\partial"]   & \cdots
\end{tikzcd}
\end{equation*}
\item For   $A\in \DMod(\widehat{\Z}\llbracket \widehat G \rrbracket)$, $A$ can be viewed as a $\widehat{\Z}[G]$-module, and there is a natural homomorphism $\mathbf{H}^\ast(\widehat{G},\overline{\mathcal{H}};A)\to H^\ast(G,\mathcal{H};A)$ of $\widehat{\Z}$-modules induced by $\iota$ such that the following diagram commutes. 
\begin{equation*}
\begin{tikzcd}
\cdots \arrow[r] & {\mathbf{H}^{k}(\widehat{G},\overline{\mathcal{H}};A)} \arrow[r] \arrow[d, "\iota^\ast"] & \mathbf{H}^k(\widehat{G};A) \arrow[r] \arrow[d, "\iota^\ast"] & \mathop{\oplus}\limits_{i=1}^{n} \mathbf{H}^k(\overline{H_i};A) \arrow[r] \arrow[d, "(\iota|_{H_i})^\ast"] & {\mathbf{H}^{k+1}(\widehat{G},\overline{\mathcal{H}};A)} \arrow[d, "\iota^\ast"] \arrow[r] & \cdots \\
\cdots \arrow[r] & {H^{k}(G,\mathcal{H};A)} \arrow[r]                                               & H^k(G;A) \arrow[r]                                    & \mathop{\oplus}\limits_{i=1}^{n} H^k(H_i;A) \arrow[r]                                              & {H^{k+1}(G,\mathcal{H};A)} \arrow[r]                                               & \cdots
\end{tikzcd}
\end{equation*}
\end{enumerate}
\end{prop}

\subsubsection{Goodness}

\begin{defn}
    A finitely generated group $G$ is said to be \emph{cohomologically good in the sense of Serre} if, for any finite discrete $\widehat{G}$-module $M$, the natural map $G \to \widehat{G}$ to its profinite completion induces an isomorphism $\mathbf{H}^*(\widehat{G},M) \to H^*(G,M)$.
\end{defn}
Where this does not cause confusion, we will simply say that $G$ is good. As explained in the introduction, there are many examples of good groups. The relationship between good groups and their profinite completions is particularly intimate. As has already been observed in \cite{KW16}, the cup product is defined for cohomology of profinite groups by precisely the same formulae as for abstract groups. Thus for a good group $G$ and a finite discrete $\widehat{G}$-module $M$, profinite completion $\iota: G \to \widehat{G}$ will induce an isomorphism of graded algebras $\iota^*: \mathbf{H}^*(\widehat{G},M) \to H^*(G,M)$ under the cup product. 

The condition on cohomology also implies the corresponding condition on homology.
\begin{prop} \cite[Proposition 6.2]{Wilkes2019} \label{goodhomology}
    Let $G$ be a good group of type $FP_{\infty}$. Then for any finite module $M$, the map $\iota_k: H_k(G,M) \to \mathbf{H}_k(\widehat{G},M)$, induced by profinitely completing $G$, is an isomorphism.
\end{prop}
Furthermore, resolutions of $\Z$ as a $\Z[G]$-module give rise to resolutions of $\widehat{\Z}$ as a $\widehat{\Z} \llbracket \widehat{G} \rrbracket $-module.
\begin{prop}~\cite[Proposition~3.1]{JaikinZapirain2020} \label{completingresolution}
    Let $G$ be a good group of type $\FP$. Given a finite type projective resolution $F_\bullet$ of $\Z$ as a $\Z[G]$-module, profinite completion gives rise to a finite type projective resolution $\widehat{F_\bullet}$ of $\widehat{\Z}$ as a module over the completed group ring $\widehat{\Z} \llbracket \widehat{G} \rrbracket $.
\end{prop}

There is also the corresponding notion of goodness of a pair.

\begin{defn}
    A discrete group pair $(G,\mathcal{H})$ is {\em cohomologically good}, or just good, if, for any $A\in \FMod(\Z [G])=\FMod(\widehat{\Z} \llbracket \widehat{G} \rrbracket )$ and every $k\geq0$, the map  $\iota^k: \mathbf{H}^k( \widehat{G},\overline{\mathcal{H}};A)\to H^k(G,\mathcal{H};A)$ is an isomorphism.
\end{defn}
As in the absolute case, there are also many examples of good pairs.
\begin{prop} \cite[Proposition 5.15]{Xu2025}
    Let $(G,\mathcal{H}=\{H_i\}_{i=1}^{n})$ be a discrete group pair. Suppose that $G$ is good, each $H_i$ is good, and the profinite topology on $G$ induces the full profinite topology on each $H_i$. Then the discrete group pair $(G,\mathcal{H})$ is good.
\end{prop}
\subsubsection{Duality}
We recall the notion of Poincar\'e duality groups, which are groups whose cohomology satisfies the same duality conditions as the cohomology rings of closed manifolds.
\begin{defn}
    A discrete group $G$ is said to be a \emph{Poincar\'e duality group of dimension $n$} over a ring $R$, or a $\mathsf{PD}^n(R)$-group, if it satisfies the following:
    \begin{enumerate}
    \item $\cd_R(G)=n$;
    \item $G$ is of type $\mathsf {FP}_{\infty}(R)$;
    \item $H^k(G; RG)=R$ if $k=n$ and is $0$ otherwise.
    \end{enumerate}
    We say $G$ is \emph{orientable} if the $G$-action on $H^n(G; RG)$ is trivial, and non-orientable otherwise.  We call $H^n(G;RG)$ the dualising module for $G$ over $R$.  
\end{defn}
For a group of type $\mathsf {FP}$ the third condition is equivalent to the existence of a \emph{fundamental class} $[G] \in H_n(G; RG)$ such that cap product with $[G]$ induces natural isomorphisms $H_k(G, M) \cong H^{n-k}(G, H^n(G; RG) \otimes M)$ for all $RG$-modules $M$ \cite[Chapter VIII, Theorem 10,1]{Brown1982}.

Although we will not need any of them, we remark for completeness that there are multiple notions of profinite Poincar\'e duality groups. See for instance \cite{Neukirch2008} and \cite{Symonds2000}.

\begin{prop} \label{profinite_completion_of_pdn}
    Let $p$ be a prime and let $G$ be a good $\mathsf{PD}^n(\FF_p)$-group.  The following conclusions hold:
    \begin{enumerate}
        \item \label{algebraiso} The inclusion $\iota\colon G\to \widehat G$ induces an $\FF_p$-algebra isomorphism \[\iota^\ast\colon \mathbf{H}^\ast(\widehat G;\FF_p)\to H^\ast(G;\FF_p).\]
        \item \label{fundamentalclass}The inclusion $\iota\colon G\to \widehat G$ induces isomorphisms \[\iota_n\colon H_n(G;\FF_p) \to \mathbf{H}_n(\widehat G;\FF_p).\]  Moreover, the fundamental class $[G]\in H_n(G;\FF_2)$ of $G$ is identified with the fundamental class $[\widehat G]\in \mathbf{H}_n(\widehat G;\FF_2)$. 
        \item \label{capproducts}For all $0 \leq k \leq n$, there is a commutative diagram of cap products
        \[\begin{tikzcd}
            H^{n-k}(G;\FF_2) \arrow[r, "{\langle [G],-\rangle}"]                                        & H_k(G;\FF_2) \arrow[d, "\iota_0"] \\
            \mathbf{H}^{n-k}(\widehat G;\FF_2) \arrow[u, "\iota^n"] \arrow[r, "{\langle [\widehat G],-\rangle}"] & \mathbf{H}_k(\widehat G;\FF_2)          
        \end{tikzcd}\]
        where the vertical arrows are isomorphisms.
    \end{enumerate}
\end{prop}

\begin{proof}
    \cref{algebraiso} and   \cref{capproducts}: This is a consequence of the discussion in the section on goodness. The formulae exhibited in \cite[Section 2] {Pletch80a} for the cup and cap products in cohomology of a profinite group agree with those in the discrete case.
    
    \cref{fundamentalclass}: The isomorphism follows from \cref{goodhomology}. To see that the fundamental classes are identified, note that
    \[H_n(G;\FF_2) \cong \FF_2 \cong \mathbf{H}_n(\widehat G;\FF_2)\]
    and the fundamental classes must be the unique non-zero class.
\end{proof}


\subsection{Quadratic forms}\label{sec:prelims:quadratic}
Let $G$ be a Poincar\'e duality group of dimension $n=4m$ over a commutative PID $R$. Poincar\'e duality implies that the cup product pairing on the torsion-free part of $H^{2m}( G, R)$ is symmetric, and defines a quadratic form on this $R$-module. We will consider to what extent this can be detected in the profinite completion. This can be considered as a question of local-to-global lifting, i.e. whether two quadratic forms defined $\Z$ that are equivalent over all completions of $\Z$ are in fact equivalent over $\QQ$. An answer is provided by the celebrated Hasse-Minkowski theorem. We will use a slightly modified version, which is explained in \cite[Chapter 15, Section 5]{Conway1999}.

Given a rational or $p$-adic integer $n$, it can be written uniquely in the form $n=p^a \cdot b$ where $\gcd(p,b)=1$. $p^a$ is called the \emph{$p$-part} of $n$ and $b$ is called the \emph{$p'$-part} of $n$. 
For $p=2$, denote by $u_i$ any $2$-adic unit congruent to $i \mod 8$. A \emph{$2$-adic antisquare} is a number of the form $2^{2k+1} \cdot u_{\pm 3}$, $k \in \Z$.
For $p>2$, denote by $u_+$ (resp. $u_-$) any $p$-adic unit which is (resp. is not) a quadratic residue $ \mod p$. A \emph{$p$-adic antisquare} is a number of the form $p^{2k+1} \cdot u_-$, $k \in \Z$.
Given an integral quadratic form $f= \mathrm{diag} \{p^{a_1} \cdot b_1, p^{a_2} \cdot b_2, \dots \}$, where $m$ of these numbers are $p$-adic antisquares, its \emph{$p$-signature} is defined to be the residue class modulo 8 of
\begin{equation}
    \begin{cases}
        4m+ \sum_i{p^{a_i}} &p \neq 2; \\
        4m+ \sum_ib_i.
    \end{cases}
\end{equation}
The $2$-signature is also called the \emph{oddity} of $f$. We further define the \emph{$p$-excess} to be
\begin{equation*}
        \begin{cases}
        (p\text{-signature}) - \text{dimension} &p \neq 2; \\
        \text{dimension}  - (p\text{-signature}) & p=2.
    \end{cases}
\end{equation*}
These are all invariants of the quadratic form. The analogue of the Hilbert reciprocity law in this context is 
\begin{equation} \label{sum_formula}
    \text{signature}(f)+ \sum_{p \geq 3} p\text{-excess}\equiv \text{oddity}(f) \mod 8.
\end{equation}
The version of the Hasse-Minkowski principle we will use is:
\begin{thm} \label{Hasse} \cite[Chapter 15, Section 5, Theorem 3]{Conway1999}
    Two nonsingular integral quadratic forms of the same dimension are equivalent over $\QQ$ if and only if both of the following hold:

    \begin{enumerate}
        \item the quotient of their determinants is a rational square;
        \item they have the same signature, the same oddity, and, for all $p \geq 3$, the same $p$-excesses modulo 8
    \end{enumerate}
\end{thm}

\section{Profinite Steenrod squares}\label{sec:Steenrod}
In this section we recall the formulation of Steenrod squares over general module categories from \cite{Epstein}.
Let $\calm$ be the category of modules over a ring $R$ of characteristic $p$.  Suppose $R$ has a commutative associative diagonal $\Delta\colon R\to R\otimes R$ and $\calm$ has a tensor product.  Given $R$-modules $A,B\in \calm$ we consider their tensor product as abelian groups $A\otimes_\Z B$ as an $R$-module via the diagonal $\Delta$.

Let $\Ab$ denote the category of abelian groups equipped with its usual tensor product $-\otimes_\Z-$.  There are adjoint functors
\[\mathrm{frgt}\colon \calm \to \Ab \quad\textrm{and}\quad \mathrm{indt}\colon \Ab \to \calm,\]
where $\mathrm{frgt}$ is the forgetful functor, and for $A\in\Ab$ we have $\mathrm{indt}(A)=R \otimes_\Z A$.

\begin{thm}[Epstein]
    If $A$ is a commutative associative algebra in $\calm$ and $C$ is a commutative associative coalgebra in $\calm$, then $\Ext_R^\ast(C,A)$ admits reduced power operations satisfying the following properties:
    \begin{enumerate}
        \item[($p=2$)]
        \begin{enumerate}
            \item $\Sq^i\colon \Ext^n(C,A)\to \Ext^{n+i}(C,A)$ is a homomorphism;
            \item $\Sq^n\colon\Ext^n(C,A)\to  \Ext^{2n}(C,A)$ is given by $x\mapsto x^2$;
            \item if $n>i$, then $\Sq^i\colon\Ext^n(C,A)\to \Ext^{n+i}(C,A)$ is the zero map;
            \item (Cartan formulae) $\Sq^n(xy)=\sum_{i+j=n}\Sq^i(x)\Sq^j(y)$;
            \item (Adem's relations) suppose $\Sq^i=0$ for $i<0$ in $\Ext_R^\ast(C,A)$, if $a<2b$ then
            \[\Sq^a\Sq^b=\sum_{j}\binom{b-1-j}{a-2j}\Sq^{a+b-j}\Sq^j\]
            \item if $i<0$, then $\Sq^i=0$.
        \end{enumerate}
        \item[($p>2$)] there exist power operations $P^i$ and $Q^i$ satisfying
        \begin{enumerate}[resume]
            \item \begin{align*}
                &P^i\colon \Ext^n(C,A)\to \Ext^{n+2i(p-1)}(C,A) \text{ and} \\
                &Q^i\colon \Ext^n(C,A)\to \Ext^{n+2i(p-1)+1}(C,A)
            \end{align*}
            are homomorphisms
            \item $P^n \colon\Ext^{2n}(C,A)\to  \Ext^{2np}(C,A)$ is given by $x\mapsto x^p$
            \item if $n>2i$, then $P^i$ is the zero map, and if $n \geq 2i$ then $Q^i$ is the zero map;
            \item (Cartan formulae) \begin{align*}
                &P^n(xy)=\sum_{i+j=n}P^i(x)P^j(y); \\
                &Q^n(xy)=\sum_{i+j=n}Q^i(x)P^j(y) + (-1)^{\rm{dim}(x)}\sum_{i+j=n}P^i(x)Q^j(y)
            \end{align*} 
            \item if $i<0$, then $P^i,Q^i=0$.
            \item \begin{enumerate}
    \item If \( a < pb \) then  
    \[P^a P^b = \sum_t (-1)^{a+t} \binom{(b-1)(b-t)-1}{a-pt} P^{a+b-t} P^t. \]

    \item If \( a \leq pb \) then  
    \[
    Q^a P^b = \sum_t (-1)^{a+t} \binom{(p-1)(b-t)-1}{a-pt} Q^{a+b-t} P^t.
    \]

    \item If \( a \leq pb \) then  
    \[
    Q^a Q^b = \sum_t (-1)^{a+t+1} \binom{(p-1)(b-t)-1}{a-pt-1} Q^{a+b-t} Q^t.
    \]

    \item If \( a < pb \) then  
    \[
    Q^a P^b = \sum_t (-1)^{a+t} \binom{(p-1)(b-t)-1}{a-pt} Q^{a+b-t} P^t.
    \]
\end{enumerate}
        \end{enumerate}
    \end{enumerate}
\end{thm}


We will now recall the definitions of the Steenrod squares and prime power operations in the case of a discrete or profinite group $G$ from \cite{Epstein}.   This is somewhat involved and is only used in this section; we also only recall the parts of op. cit. that we need and so the attentive reader is encouraged to go read Epstein's beautiful paper to see the proofs.  We recommend the reader at least look at \cite[\S11.2]{Epstein} to see how the special case of modules over a ring of Epstein's general construction we are using applies.

Let $R$ be a ring. Recall that, to avoid confusion, we denote by $\Z/p\Z$ the cyclic group of order $p$ and by $\FF_p$ the field of $p$ elements. Let $A$ be an $R$-algebra (or a sheaf of $R$-algebras) over $\FF_p$, which has a multiplication map \[m: A^{\otimes p} \to A. \] 
Denote by $\Ext_{R,\Z/p\Z}^*( -, -)$ the functor $\Ext$ taken in the category of (sheaves of) algebras with a $\Z/p\Z$-action. The group $\Z/p\Z$ acts on $A^{\otimes p}$ by cyclically permuting the factors.  By \cite[\S 5.1.2]{Epstein}, this gives rise to a map (not in general a homomorphism)
\[P: \Ext^q_{R}(\FF_p, A) \to \Ext_{R,\Z/p\Z}^{pq}(\FF_p, A^{\otimes p}).\]
Note that in our case $m_\ast P$ actually is a homomorphism \cite[Lemma 6.6]{Epstein}.  Furthermore, by \cite[\S4.4.4]{Epstein}, since  $\Z/p\Z$ acts trivially on both $A$ and $\Z/p\Z$, there is an isomorphism
\[\varphi: H^*(\Z/p\Z; \FF_p) \otimes \Ext^*_{R}(\FF_p, A) \to \Ext_{R,\Z/p\Z}^\ast(\FF_p, A)\]
which is natural with respect to ring homomorphisms $R\to S$.
Let $u\in \Ext^q_{\Z}(\FF_p, A)$.  Following \cite[Definition 6.1]{Epstein}, we define
$D_k: \Ext^q_{R}(\FF_p, A)\to \Ext^{pq-k}_{R}(\FF_p, A)$ 
by
$$\varphi^{-1} m_*P(u)= \sum_iw_k \otimes D_k(u)$$
where $w_k\in H^*(\Z/p\Z; \FF_p)$ is an element defined in \cite[V5.2]{Steenrod1962}.  By \cite[Corollary~6.7]{Epstein}, the maps $D_k$ are homomorphisms.  We are now ready to define the operations following \cite[\S7.1]{Epstein}.  Let $u\in \Ext^q_{\Z}(\FF_p, A)$, if $p>2$ let $m=(p-1)/2$, and let $r=i+m(q^2-q)/2$.
\begin{enumerate}
    \item $p=2$.  If $q\geq i$ we define $\Sq^i (u)=D_{q-i}(u)$.
    \item $p>2$. \begin{enumerate}
        \item If $q\geq 2i$ we define $P^i(u)=(-1)^r(m!)^{-q} D_{(q-2i)(p-1)}(u)$.  If $q<2i$ we define $P^i(u)=0$.
        \item If $q>2i$ we define $Q^i(u)=(-1)^{r+1}(m!)^{-q}D_{(q-2i)(p-1)}(u)$.  If $q\leq 2i$ we define $Q^i(u)=0$.
    \end{enumerate}
\end{enumerate}

\begin{prop} \label{naturalsq}
    Let $p$ be a prime and let $G$ be a finitely generated group. For any finite algebra $A$ over $\FF_p$, the inclusion $\iota: G \to \widehat{G}$ induces maps on cohomology $\iota^*$ such that, for all $q$ and $i$, the following diagrams commute
    \begin{enumerate}
        \item $(p=2)$
    \begin{equation*}
        \begin{tikzcd}
            \mathbf{H}^{q}(\widehat G; A) \arrow[r, "\Sq^i"]  \arrow[d, "\iota^*"]  & \mathbf{H}^{q+i}(\widehat G; A) \arrow[d, "\iota^*"]  \\
            H^{q}(G; A) \arrow[r, "\Sq^i"]     & H^{q+i}( G; A) 
        \end{tikzcd}
    \end{equation*}
    
        \item $(p>2)$
        \begin{enumerate}
        \item 
        \begin{equation*}
        \begin{tikzcd}
            \mathbf{H}^{q}(\widehat G; A) \arrow[r, "P^i"]  \arrow[d, "\iota^*"]  & \mathbf{H}^{q+2i(p-1)}(\widehat G; A) \arrow[d, "\iota^*"]  \\
            H^{q}(G; A) \arrow[r, "P^i"]     & H^{q+2i(p-1)}( G; A) 
        \end{tikzcd}
    \end{equation*}
    
    \item 
    \begin{equation*}
        \begin{tikzcd}
            \mathbf{H}^{q}(\widehat G; A) \arrow[r, "Q^i"]  \arrow[d, "\iota^*"]  & \mathbf{H}^{q+2i(p-1)+1}(\widehat G; A) \arrow[d, "\iota^*"]  \\
            H^{q}(G; A) \arrow[r, "Q^i"]     & H^{q+2i(p-1)+1}( G; A). 
        \end{tikzcd}
    \end{equation*}
    \end{enumerate}
    \end{enumerate}
\end{prop}

\begin{proof}
    Tracing through Epstein's construction, discussed right before the statement of the proposition, one checks that all the maps which are used to define the Steenrod squares and the Steenrod prime power operations $P$ and $Q$ implicitly are valid in both the category of $\FF_p[G]$-modules and in the category of ${\FF_p \llbracket \widehat G \rrbracket }$-modules, as with the cup and cap products (since these are defined at the level of chains; see the discussion on goodness), and all of them commute with $\iota^*$. We explain this in more detail in the case $p=2$, the case $p>2$ being analogous.

    Let $R=\FF_p[G]$ the group ring, and $\widehat{R} = \FF_p \llbracket \widehat{G} \rrbracket$ its completed group ring. For clarity, we denote by maps in the category of modules over group rings discrete groups by the corresponding letters used above, and the maps in the category of modules over completed group rings of profinite groups by the corresponding letters with a hat. Consider the following diagram.
    \[
\begin{tikzcd}
\Ext^q_{\hat R}(\FF_p, A) \arrow[r, "\hat P"] \arrow[d, "\widehat \iota_q"] & \Ext^{pq}_{\widehat{R}, \Z/p}(\FF_p, A^{\otimes p}) \arrow[d, "\widehat{\iota_{pq}}"] \arrow[r, "\hat m_*"] & \Ext^{pq}_{\widehat{R}, \Z/p}(\FF_p, A^{\otimes p}) \arrow[d, "\widehat{\iota_{pq}}"] \arrow[r, "\hat \varphi^{-1}_*"]& H^*(\Z/p, \FF_p) \otimes \Ext_{\widehat{R}}(\FF_p, A)\arrow[d, "\id \otimes \widehat \iota_q"] \\
\Ext^q_{R}(\FF_p, A) \arrow[r, "P"] & \Ext^{pq}_{R, \Z/p}(\FF_p, A^{\otimes p}) \arrow[r, "m_*"] & \Ext^{pq}_{R, \Z/p}(\FF_p, A^{\otimes p}) \arrow[r, "\varphi^{-1}_*"] & H^*(\Z/p, \FF_p) \otimes \Ext_{R}(\FF_p, A) \\
\end{tikzcd}\]
By \cite[Definition 6.1]{Epstein}, when $p=2$ the composition of the top row defines $\Sq^i$ on $\mathbf{H}^{q}(\widehat G; A)$, while the composition of the bottom row defines $\Sq^i$ on $H^{q}(G; A)$. To show the proposition it suffices to show that the diagram commutes.

The leftmost square commutes because the map on right derived functors is induced by a map at the level of chains, described in \cite[5.1.2]{Epstein}). The middle square commutes by naturality in the coefficients (c.f. \cite[Section 7.3]{Epstein}). The rightmost square commutes by \cite[Theorem 4.4.4]{Epstein}). 
\end{proof}

\section{Detection of topological invariants}\label{sec:detect}

We will work in $\FMod$ over the appropriate completed group algebra until otherwise specified.
\subsection{Stiefel--Whitney classes}\label{sec:detect:SW}
In this section we prove profinite invariance of Stiefel--Whitney classes.

\subsubsection{The closed case}

\begin{thmx}\label{thmx:Invariance_of_SW_classes}
    Let $M$ and $N$ be connected aspherical Poincar\'e complexes with good fundamental groups $G$ and $H$, respectively.  If $\Theta\colon \widehat G\to\widehat H$ is an isomorphism, then there exists an $\FF_2$-algebra isomorphism $\theta\colon H^\ast(N;\FF_2)\to H^\ast(M;\FF_2)$ such that $\theta(w_i(N))=w_i(M)$ and $\theta(w(N))=w(M)$.
\end{thmx}
\begin{proof}
    Since $M$ and $N$ are aspherical the classifying maps $M\to BG$ and $N\to BH$ induce $\FF_2$-algebra isomorphisms
    \[H^\ast(G;\FF_2)\cong H^\ast(M;\FF_2)\text{ and }H^\ast(H;\FF_2)\cong H^\ast(N;\FF_2)\]
     which preserve Steenrod squares.
    Now, since $G$ and $H$ are good, the natural inclusions $G\to \widehat G$ and $H\to \widehat H$ induce natural $\FF_2$-algebra isomorphisms
    \[\mathbf{H}^\ast(\widehat G;\FF_2)\cong H^\ast(G;\FF_2)\text{ and }\mathbf{H}^\ast(\widehat H;\FF_2)\cong H^\ast(H;\FF_2),\]
    which by \cref{naturalsq} preserve Steenrod squares.  Combining these isomorphisms we obtain an $\FF_2$-algebra isomorphism $\theta\colon H^\ast(N;\FF_2)\to H^\ast(M;\FF_2)$ that preserves Steenrod squares.  Now, the Wu classes are completely determined in terms of the Steenrod squares, so we have  $\theta(v_k(N))=v_k(M)$ and $\theta(v(N))=v(M)$.  By Wu's Theorem, the Stiefel--Whitney classes are completely determined by the Wu classes (\Cref{WuTheorem}).  Thus, we have $\theta(w_k(N))=w_k(M)$ and $\theta(w(N))=w(M)$ as required.
\end{proof}

\subsubsection{The relative case}
\begin{thm} \label{equal_relative_SWclasses}
    Let $(K_1, L_1)$ and $(K_2, L_2)$ be connected aspherical relative Poincar\'e complexes of dimension $n$ with good fundamental groups $G_1$ and $G_2$. Suppose further that, for $j=1,2$, the inclusion of each connected component $C_{i,j}$ of $L_{j} \hookrightarrow K_j$ is $\pi_1$-injective with image $H_{i,j}$. Let $\mathcal{H}_j$ be the union of all $H_{i,j}$. Suppose that both $(G_1, \calh_1)$ and $(G_2, \calh _2)$ are good pairs. If $\Theta\colon (\widehat G_2, \overline{\mathcal{H}_2}) \to (\widehat G_1, \overline{\mathcal{H}_1})$ is an isomorphism of pairs, then the induced map $\theta\colon H^\ast(G_1, \mathcal{H}_1;\FF_2)\to H^\ast(G_2, \mathcal{H}_2;\FF_2)$ maps the relative Stiefel--Whitney classes of $(K_1, L_1)$ to those of $(K_2, L_2)$.
\end{thm}
\begin{proof}
    
    The proof is almost identical to that of \cref{thmx:Invariance_of_SW_classes}. The same argument shows that there is an $\FF_2$-algebra isomorphism $\theta\colon H^\ast(G_1, \mathcal{H}_1;\FF_2)\to H^\ast(G_2, \mathcal{H}_2;\FF_2)$. Since the pair $(G_1, \mathcal{H}_1)$ is good, each Wu class $(v_{(\widehat G_1, \overline{\mathcal{H}_1})})_k$ for the pair $(\widehat G_1, \overline{\mathcal{H}_1})$ must get mapped to the Wu class $(v_{(G_1, \mathcal{H}_1)})_k$ for the pair $(G_1, \mathcal{H}_1)$ by naturality of the Steenrod squares. This implies that $\theta ((v_{(G_1, \mathcal{H}_1)})) =(v_{(G_2, \mathcal{H}_2)})$. Since $\theta$ commutes with the action of Steenrod squares, it sends Stiefel--Whitney classes to Stiefel--Whitney classes as well.
\end{proof}

\subsection{Spin Structures}\label{sec:detect:spin}

\begin{corollary}\label{spin}
    Let $M$ and $N$ be smooth closed connected aspherical manifolds with good fundamental groups $G$ and $H$, respectively.  If $\Theta\colon \widehat G\to\widehat H$ is an isomorphism, then $M$ admits a spin structure if and only if $N$ admits a spin structure.
\end{corollary}
\begin{proof}
    It suffices to prove $w_1(M)=w_2(M)=0$ if and only if $w_1(N)=w_2(N)=0$.  But this follows from \Cref{thmx:Invariance_of_SW_classes}.
\end{proof}

Knowledge of the profinite invariance of spin structures can be applied quite effectively to profinite classification problems.  We describe a simple example in the context of flat $6$-manifolds to demonstrate the idea.

\begin{example}
    There are $4$ flat $6$-manifolds in the $\QQ$-class with \texttt{CARAT} identifier \texttt{min.368} with \texttt{CARAT} identifiers \texttt{min.368.1.1.8}, \texttt{min.368.1.1.10}, \texttt{min.368.1.1.15}, and \texttt{min.368.1.4.4}.  Denote these by $M_1,\dots,M_4$ respectively.  These flat manifolds have holonomy $C_3\times D_4$ where $|C_3|=3$ and $|D_4|=8$. By \cite[Proposition~1.2]{CHMV2025}, any finitely generated residually finite group $G$ whose profinite completion is isomorphic to $\widehat{\pi_1(M_i)}$ for some $i$ is the fundamental group of a flat manifold with holonomy in the same $\QQ$-class.  In particular, $G$ is isomorphic to some $\pi_1 (M_i)$ for some $i$.  Exactly one of these four manifolds, namely $M_3$, is spin \cite{LutowskiPutrycz2015} (the data we are using is available at \cite{LutwoskiData}).  Thus, if $\widehat G\cong \widehat{\pi_1M}_{3}$, then $G\cong \pi_1M_3$.
\end{example}

We record the analogous statement in the case with boundary; the proof is exactly the same as in the closed case.

\begin{corollary}
    Let $(M,\partial M)$ and $(N,\partial N)$ be compact connected aspherical smooth manifolds with boundary. Suppose that the inclusion of the fundamental groups of boundary components gives rise to good group pairs $(G_1,\calh_1)$ and $(G_2,\calh_2)$, respectively. If there exists an isomorphism of pairs $\Theta\colon  (\widehat G_1,\overline\calh_1)\to (\widehat G_2,\overline \calh_2)$, then $M$ admits a spin structure if and only if $N$ admits a spin structure.
\end{corollary}

\subsection{Stiefel--Whitney numbers}\label{sec:detect:SWnums}

\begin{thm} \label{equalSWnumbers}
    Let $M$ and $N$ be connected aspherical Poincar\'e complexes with good fundamental groups $G$ and $H$.  If $\Theta\colon \widehat G\to\widehat H$ is an isomorphism, then the Stiefel--Whitney numbers of $M$ and $N$ are equal.
\end{thm}
\begin{proof}
    \cref{thmx:Invariance_of_SW_classes} implies that $G$ and $\widehat G$ have the same Stiefel--Whitney classes. By \cref{profinite_completion_of_pdn} the cup and cap products in $G$ and $\widehat G$ agree and the fundamental classes are identified under the map on cohomology induced by profinite completion, so they have the same Stiefel--Whitney numbers. Since the same is true of $H$ and $\widehat H \cong \widehat G$, we conclude.
\end{proof}

\begin{corollary}\label{cobordant}
    Let $M$ and $N$ be closed connected aspherical smooth manifolds with good fundamental groups $G$ and $H$.  If $\Theta\colon \widehat G\to\widehat H$ is an isomorphism, then $M$ and $N$ are cobordant.
\end{corollary}
\begin{proof}
    This is an immediate consequence of \cref{equalSWnumbers} and Thom's result that the manifolds in question are cobordant if and only if they have the same Stiefel--Whitney numbers.
\end{proof}

\subsection{Integral Bocksteins}\label{sec:detect:Bocksteins}
From now on we work in the category $\PMod$ over the appropriate completed group algebra.
\begin{lemma} \label{profinite_completion_of_cohomology}
    Let $G$ be a good discrete group of type $\FP$.  Let $\phi: \Z \to \widehat{\Z}$ be the profinite completion map, and for each prime $p$ let $\phi_p: \Z \to {\Z}_p$ be the pro-$p$ completion map.
    Changing coefficients using $\phi$ and $\phi_p$ respectively gives rise to the maps 
    \begin{align}
        \phi_*: H^*(G; \Z) &\to H^*(G, \widehat{\Z}) \\
        (\phi_p)_*: H^*(G; \Z) &\to H^*(G; \Z_p)
    \end{align}
    $ $
    Furthermore, $\phi_*$ and $(\phi_p)_*$ can be canonically identified with the profinite and pro-$p$ completion maps respectively, which gives rise to isomorphisms $\widehat{H^*(G; \Z)} \cong H^*(G, \widehat{\Z})$ and ${H^*(G; \Z_p)} \cong (H^*(G; \Z))_p$ 
\end{lemma}
\begin{proof}
We give the proof in the case of profinite completion. The pro-$p$ completion case is analogous. $\phi$ induces the map 
 \begin{align*}
      \overline{\phi}:\Hom(H_i(G; \Z), \Z) &\to \Hom(H_i(G;{\Z}),\widehat{\Z}), \\
      \psi &\mapsto \phi \circ \psi. 
 \end{align*}
 
The Universal Coefficient Theorem for cohomology is natural \emph{in the coefficients}, so the following diagram commutes:
\begin{equation}\label{eqn:Ext-ension}
    \begin{tikzcd}[scale cd=0.9] 
      0 \ar[r] & \Ext_{\Z}^1(H_{i-1}(G;\Z), \Z) \ar[r] \ar[d] & H^i(G; \Z) \ar[r] \ar[d, "\phi_*"] & \Hom(H_i(G; \Z), \Z) \ar[r] \ar[d, "\overline{\phi}"] & 0 \\
      0 \ar[r] & \Ext_{\Z}^1(H_{i-1}(G;\Z), \widehat{\Z})\ar[r] & H^i(G; \widehat{\Z}) \ar[r] &\Hom(H_i(G; \Z), \widehat{\Z})\ar[r] & 0
    \end{tikzcd}
\end{equation}
The vertical arrows are the maps given by change of coefficients. Now note that, since $G$ is of type $\FP$, $H_i(G; \Z)$ is a finitely generated abelian group. Hence, for any torsion-free abelian group $A$, $\Hom(H_i(G; \Z), A) \cong A^{b_i(G)}$, and for any groups $G, H$, $\widehat{G \times H} \cong \widehat{G} \times \widehat{H}$, so $\overline{\phi}$ is in fact profinite completion.

Recall that for any group $A$, $\Ext_{\Z}^1(\Z, A)=0$, and 
\[
    \Ext_{R}^{i}\left(\bigoplus _{\alpha }M_{\alpha },N\right) \cong \prod _{\alpha }\Ext_{R}^{i}(M_{\alpha },N).
\]
 Furthermore, whenever $R$ is a commutative ring and $u \in R$ is not a zero divisor, then
\[\Ext_{R}^{i}(R/(u), B)\cong {
\begin{cases}B[u]&i=0;\\
B/uB&i=1; \\ 
0&{\text{otherwise.}}\end{cases}}\]
Suppose that $H_{i-1}(G;\Z) = \mathbb{Z}^k \oplus \sum_{j=1}^m \Z/ n_j\Z$ (it must be of this form as the homology groups are finitely generated abelian groups). For each summand $\Z/n_j \Z$ of $H_{i-1}(G;\Z)$, we have $\Ext_{\Z}^1( \Z/ n_j\Z, \Z) \cong  \Z/ n_j\Z$. Hence 
\begin{align*}
    \Ext_{\Z}^1( H_{i-1}(G;\Z) , \Z) & \cong \Ext_{\Z}^1\left(\mathbb{Z}^k \oplus \bigoplus_{j=1}^m \Z/ n_j\Z, \Z\right) \\
    &\cong (\Ext_{\Z}^1( \Z, \Z))^k \times \prod_{j=1}^m\Ext_{\Z}^1( \Z/ n_j\Z, \Z) \\
    &\cong \prod_{j=1}^{m}  \Z/ n_j\Z.
\end{align*}
Similarly,
\begin{align*}
    \Ext_{\Z}^1(H_{i-1}(G;\Z), \widehat{\Z})&\cong   \Ext_{\Z}^1\left(\mathbb{Z}^k \oplus \bigoplus_{j=1}^m \Z/ n_j\Z, \widehat{\Z}\right)\\ 
    &\cong \prod_{j=1}^m \widehat{\Z}/n_j \widehat{\Z} \\ 
    &\cong \prod_{j=1}^m\varprojlim_k \left((\Z/k\Z)/n_j(\Z/k\Z)\right) \\ 
    &\cong \prod_{j=1}^m \Z/n_j\Z.
\end{align*}
By naturality, the left hand arrow of \eqref{eqn:Ext-ension} is the identity.

Now, $H^i(G; \Z) \cong \Z^{b_i(G)} \oplus T$, where $T =\Ext_{\Z}^1( H_{i-1}(G;\Z) , \Z)$ is a finite abelian group, and $H^i(G; \widehat \Z) \cong (\widehat \Z)^{b_i(G)} \oplus T$. By naturality of UCT, $\phi_*$ is given by the identity on the $T$ factor and $\overline \phi$ on the $\Z^{b_i(G)}$ factor, which is the profinite completion map. But this is precisely the profinite completion map.
\end{proof}

    \begin{prop} \label{profinite_algebra_iso}
         Let $G$ be a good discrete group of type $\FP_{\infty}$, and let $\widehat{\Z}$ be the $\widehat{G}$-module with trivial $\widehat{G}$-action. Then, the maps
 \begin{align} \label{Zhatiso}
       &\iota^*: \mathbf{H}^*(\widehat{G};\widehat{\Z}) \to H^* (G;\widehat{\Z}) \\
       &\iota^*: \mathbf{H}^*(\widehat{G};{\Z_p}) \to H^* (G; {\Z_p})
 \end{align}
induced by $\iota: G\to \widehat G$ is an isomorphism of $\widehat{\Z}$, resp. $\Z_p$-algebras, and these isomorphisms are natural in $G$.  Moreover, 
\begin{align*}
    H^* (G;\widehat{\Z}) &\cong H^\ast(G;\Z)\otimes_\Z \widehat \Z \cong \widehat{H^\ast(G;\Z)} \quad \text{ and}\\
    H^* (G;\Z_p) &\cong H^\ast(G;\Z)\otimes_\Z \Z_{ p} \cong \widehat{(H^\ast(G;\Z))}_{p} .
\end{align*}
    \end{prop}    
The proof is almost identical to \cite[Proposition 5.19]{Xu2025}. There is a generalisation to the relative case, but, as in \cite[Proposition 5.19]{Xu2025}, the proofs of the relative and absolute case are analogous. 
\begin{proof}
As above, we do the case of $\widehat{\Z}$, the $\Z_p$ case being analogous. We show that 
\begin{equation} \label{changecoefficients}
    \phi: H^*(G, \widehat{\Z}) \to \varprojlim H^*(G; \Z/n!)
\end{equation}
is an isomorphism.

Since $G$ is of type $\mathsf{FP}_\infty$, by \cite[Chapter VIII, Proposition 4.5]{Brown1982}, we can take a free resolution $F_\bullet\to \Z \to 0$ of finite type in the category of right $\Z[G]$-modules. Let $C^n_\bullet=\Hom_{\Z[G]}(F_\bullet, \Z/n!)$. Then,  $C^0_\bullet\leftarrow C^1_\bullet \leftarrow C^2_\bullet \leftarrow \cdots$ is a tower of chain complexes of abelian groups. 

The functor $\Hom_{\Z[G]}(F_\bullet, -)$ is exact because each $F_k$ is a free $\Z[G]$-module. Now, note that $\Z/(n+1)!\to \Z/n!$ is surjective, so $C^{n+1}_k \to C^n_k$ is surjective for any $n,k\in \mathbb{N}$. Consequently, the tower of cochain complexes $C^0_\bullet\leftarrow C^1_\bullet  \leftarrow \cdots$ satisfies the Mittag-Leffler condition degree-wise. By \cite[Theorem 3.5.8]{Weibel94}, or more precisely the variant stated after it, there is a short exact sequence 
\[
    0\to {\varprojlim\limits_n}^1 H^{k-1} (C^n_{\bullet})\to H^k (\varprojlim\limits_n C^n_\bullet) \to \varprojlim\limits_n H^k (C^n_\bullet)\to 0.
\]
By construction 
\[
    H^k(C^n_\bullet) \cong  \Ext^k_{\Z G}(\Z, \Z/n!) \cong H^{k}(G; \Z/n!).
\]
Letting $F_k= \Z[G]^{\oplus r_k}$, we obtain $C^n_k = \Hom_{\Z[G]}(\Z[G]^{\oplus r_k}, (\Z/n!)) \cong (\Z/n!)^{\oplus r_k}$. Thus,
\[
    \varprojlim_n C^n_k\cong \widehat{\Z}^{\oplus r_k}\cong \Hom_{\Z[G]}(F_k,  \widehat{\Z} ),
\]
and 
\[
    H^k(\varprojlim_n C^n_\bullet)\cong H^k(\Hom_{\Z[G]}(F_\bullet,  \widehat{\Z}))= \Ext^k_{\Z [G]}({\Z}, \widehat{\Z})= H^{k}(G;\widehat{\Z}). 
\]
Thus, we actually have a short exact sequence
\[
    0\to {\varprojlim\limits_n}^1 H^{k-1} (C^n_{\bullet})\to H^k(G;\widehat{\Z}) \xrightarrow{\;\phi\;} \varprojlim\limits_n H^k(G;\Z/n!)\to 0.
\]
The map $\phi$ is exactly the one in \cref{changecoefficients}, since both are induced by the maps between the coefficient modules $\widehat{\Z}\to \Z/n!$.

Note that for each $n,k\in \mathbb{N}$, $C^n_k\cong (\Z/n!)^{\oplus r_k}$ is finite, so $H_k(C^n_\bullet)$ is also finite. As a consequence, the tower of abelian groups $H_k(C^n_\bullet)$ satisfies the Mittag-Leffler condition, and so $\varprojlim^1_n H_k (C^n_\bullet)=0$ according to \cite[Exercise 3.5.2]{Weibel94}. Hence $\phi$ is an isomorphism.

Replacing occurrences of cohomology by continuous cohomology and appealing to \cref{completingresolution}, the same argument verbatim shows that 
\[
    \mathbf{H}^*(\widehat{G}; \widehat{\Z}) \to \varprojlim_n \mathbf{H}^*(\widehat{G}; \Z/n!)
\]
is an isomorphism.

Since $G$ is good and of type $\sf{FP}_\infty$, for all $n\in \mathbb{N}$ the map $\iota^*: H^*(G;\Z/n!) \to H^* (\widehat{G};\Z/n!)$ is an isomorphism. In addition, these isomorphisms are natural with respect to the coefficient modules $\Z/n!$. Thus, we have a commutative diagram of abelian groups
\begin{equation*}
\begin{tikzcd}
{\mathbf{H}^*(\widehat{G};\widehat{\Z})} \arrow[d, "\cong"',"\phi"] \arrow[r, "\iota^*"] & {H^*(G;\widehat{\Z})} \arrow[d, "\cong"] \\
{\varprojlim\limits_n \mathbf{H}^*(\widehat{G};\Z/n!)}   \arrow[r, "\cong"',"\iota^*"]   &  {\varprojlim\limits_n H^*(G;\Z/n!)}
\end{tikzcd}
\end{equation*}
where the two vertical maps and the bottom map were proved to be isomorphisms above. 

Thus, $\iota^*:  {\mathbf{H}^*(\widehat{G};\widehat{\Z})} \to {H^*(G;\widehat{\Z})}$  is also an isomorphism of abelian groups. Since $\iota^*$ is a homomorphism of $\widehat{\Z}$-modules, and all the intermediate maps respect cup products, we have an isomorphism of algebras. Since all the maps used above are natural, naturality of the isomorphism follows. 
\end{proof}

\begin{corollary}
    Let $G$ be a good discrete group of type $\FP_{\infty}$.  The following diagram commutes:
\begin{equation} \label{bigdiagram}
    \begin{tikzcd}
    & H^k(G; \Z/2\Z) \arrow[d, "id"] \arrow[r, "\beta_2"] & H^{k+1}(G; \Z) \arrow[d, hook]\\
    &H^k(G; \Z/2\Z) \arrow[r, "\beta_2"] & H^{k+1}(G; \widehat{\Z})= \widehat{H^{k+1}(G; \Z)} \\
    &\mathbf{H}^k(\widehat{G}; \Z/2\Z) \arrow[u, "\sim"] \arrow[r, "\hat{\beta_2}"] & \mathbf{H}^{k+1}(\widehat{G}; \widehat{\Z}) \arrow[u, "\sim"] 
    \end{tikzcd}
\end{equation}
The vertical maps from the first row to the second row are induced by the following commutative diagram in the coefficients:
\[ \begin{tikzcd}
    0 \ar[r] & \Z \ar[r, "\times 2"] \ar[d] & \Z \ar[r] \ar[d] & \Z/2\Z \ar[r] \ar[d, "="] & 0 \\
    0 \ar[r] &  \widehat \Z \ar[r, "\times 2"] & \widehat{\Z} \ar[r] &\Z/2\Z \ar[r] & 0
\end{tikzcd}\]
The vertical maps from the third row to the second row are induced by $\iota: G \to \widehat{G}$. Furthermore, the diagram is natural in $G$.
\end{corollary}
\begin{proof}
    The Bockstein sequence is natural in $G$, maps induced by profinite completion are natural in $G$, and the fact that \cref{Zhatiso} is natural combine to show that the diagram is natural.
\end{proof}

\subsection{Integral Stiefel--Whitney classes}\label{sec:detect:intSW}

\begin{corollary} \label{detectSW_Z}
        Let $M$ and $N$ be connected aspherical Poincar\'e complexes with good fundamental groups $G$ and $H$.  If $\Theta\colon \widehat G\to\widehat H$ is an isomorphism, then $W_k(M)=0$ if and only if $W_k(N)=0$.
\end{corollary}

\begin{proof}
    Since the diagram \eqref{bigdiagram} commutes, we obtain the following commutative diagram.

\begin{equation*}
\begin{tikzcd}
& H^k(G; \Z/2\Z)  \arrow[r, "\beta_2"] & H^{k+1}(G; \Z) \arrow[d, hook]\\
&\mathbf{H}^k(\widehat{G}; \Z/2\Z) \arrow[u, "\iota_G^k"'] \arrow[d, "\iota_H^k"] \arrow[r, "\hat{\beta_2}"] & \mathbf{H}^{k+1}(\widehat{G}; \widehat{\Z}) \\
& H^k(H; \Z/2\Z)  \arrow[r, "\beta_2"] & H^{k+1}(H; \Z) \arrow[u, hook]\\
\end{tikzcd}
\end{equation*}
    Both $\iota_G^\ast$ and $\iota_H^\ast$ preserve Stiefel--Whitney classes by \Cref{thmx:Invariance_of_SW_classes}, so \[(\iota_H^k)^{-1}(w_k(H))=(\iota_G^k)^{-1}(w_k(G)).\]
    $W_{k+1}(G) \neq 0$ if and only if $ \beta_2(\iota_G^k)^{-1}(w_k(G)) \neq 0$, and similarly $W_{k+1}(H) \neq 0$ if and only if $ \beta _2(\iota_H^k)^{-1}(w_k(H)) \neq 0$, but 
    \[\beta _2(\iota_H^k)^{-1}(w_k(H))= \beta _2(\iota_G^k)^{-1}(w_k(G)), \]
    so $W_{k+1}(G) \neq 0$ if and only if $W_{k+1}(H) \neq 0$ as required. 
\end{proof}

We can readily apply this to the problem of profinitely detecting spin$^\CC$ structures. 

\begin{corollary} \label{spinC}
        Let $M$ and $N$ be closed connected aspherical smooth manifolds with good fundamental groups $G$ and $H$.  If $\Theta\colon \widehat G\to\widehat H$ is an isomorphism, then $M$ admits a spin$^\CC$-structure if and only if $N$ does.
\end{corollary}
\begin{proof}
    It suffices to prove $w_1(M)=W_2(M)=0$ if and only if $w_1(N)=W_2(N)=0$.  But this follows from \Cref{thmx:Invariance_of_SW_classes} and \Cref{detectSW_Z}.
\end{proof}

\subsection{Pontryagin classes and numbers modulo 3}\label{sec:detect:Pmod3}
Let $M$ be a connected aspherical Poincar\'e complex. Recall that $p_k(M)$ denotes the $k^{th}$ Pontryagin class, and in this section we denote by $\overline p_k(M)$ the class of $p_k(M)$ upon reduction $\mod{3}$, and $\overline p(M)$ the class of $p(M)$ upon reduction $\mod{3}$.
\begin{thm}\label{thm:PontrayginClassesMod3}
    Let $M$ and $N$ be connected aspherical Poincar\'e complexes with good fundamental groups $G$ and $H$.  If $\Theta\colon \widehat G\to\widehat H$ is an isomorphism, then there exists an $\FF_3$-algebra isomorphism $\theta\colon H^\ast(M;\FF_3)\to H^\ast(N;\FF_3)$ such that $\theta(\overline p_i(M))=\overline p_i(N)$ and $\theta(\overline p(M))=\overline p(N)$.
\end{thm}
\begin{proof}
    The proof is analogous to the case of the Stiefel--Whitney classes.  Since $M$ and $N$ are aspherical the classifying maps $M\to BG$ and $N\to BH$ induce $\FF_3$-algebra isomorphisms
    \[H^\ast(G;\FF_3)\cong H^\ast(M;\FF_3)\text{ and }H^\ast(H;\FF_3)\cong H^\ast(N;\FF_3)\]
     which preserve the Steenrod $3$rd power operations $P^i$.
    Now, since $G$ and $H$ are good, the natural inclusions $G\to \widehat G$ and $H\to \widehat H$ induce natural $\FF_3$-algebra isomorphisms
    \[\mathbf{H}^\ast(\widehat G;\FF_3)\cong H^\ast(G;\FF_3)\text{ and }\mathbf{H}^\ast(\widehat H;\FF_3)\cong H^\ast(H;\FF_3),\]
    which by \cref{naturalsq} preserves the Steenrod $3$rd power operations $P^i$.  Now, the result follows from the Hirzebruch's algebraic description of the classes $\overline p_k$ in terms of the classes $s_3^i$ as explained in \Cref{sec:chern_pontryagin}.
\end{proof}

The following corollary follows almost identically to \Cref{equalSWnumbers} once we have argued how to orient our spaces $M$ and $N$.

\begin{corollary}\label{thm:PontrayginMod3}
    Let $M$ and $N$ be connected aspherical Poincar\'e complexes with good fundamental groups $G$ and $H$.  Then there exist $\FF_3$-orientations on $M$ and $N$ such that the Pontryagin numbers of $M$ and $N$ are equal modulo $3$.
\end{corollary}
\begin{proof}
    We have $\FF_3$-algebra isomorphisms
    \[H^\ast(M;\FF_3)\cong H^\ast(G;\FF_3)\cong H^\ast(\widehat G;\FF_3)\cong H^\ast(\widehat H;\FF_3) \cong H^\ast(H;\FF_3)\cong H^\ast(N;\FF_3)\]
    induced by profinite completions and the isomorphism between $\widehat G$ and $\widehat H$. Denote the homomorphism given by composing all isomorphisms from left to right by $\theta\colon H^\ast(M;\FF_3)\to H^\ast(N;\FF_3)$. An $\FF_3$-orientation $[M]\in H^n(M;\FF_3)$ determines a nontrivial class $\theta[M]\in H^\ast(N;\FF_3)$.  We take $[N]=\theta[M]$ to be the desired $\FF_3$-orientation of $N$.  The result now follows from naturality of the cap products with the specified orientation classes.
\end{proof}

\subsection{The intersection form}\label{sec:detect:sigma}

In this section we show that evenness the intersection form, which plays a crucial role in the topological classification of $4$-manifolds, and its signature$\mod{8}$ are profinite invariants. The residue class of the signature has been the subject of significant interest in the past; see for example \cite{Ochanine, Freedman1976}.
\begin{lemma}
       Let $M$ and $N$ be connected orientable aspherical Poincar\'e complexes of dimension $n=4m$ with good fundamental groups $G$ and $H$.  If $\Theta\colon \widehat G\to\widehat H$, then the cup product pairing on $(H^{2m}(G; \Z))^{\tf}$ is even if and only if the pairing on $(H^{2m}(H;\Z))^{\tf}$ is even. 
\end{lemma}
\begin{proof}
For $x \in (H^{2m}(G; \Z))^{\tf}$ let $\overline{x}$ denote the image of $x$ after reduction $\mod 2$, considered as a class in $(H^{2m}(G, \FF_2))$. By definition of the Wu class and the Steenrod squares, for all $x \in (H^{2m}(G; \Z))^{\tf}$ \[v_{2m} \cup \overline{x}  = \Sq^{2m}(\overline{x})= \overline{x} \cup \overline{x}.\] $\overline{x} \cup \overline{x}$ is a number in $\FF_2$ that coincides with the reduction of $x \cup x \mod{2}$, so the form is even precisely when the class $v_{2m}$ vanishes. This is detected profinitely, as shown in the proof of \cref{thmx:Invariance_of_SW_classes}.
\end{proof}

\begin{thm}\label{intersectionForm}
    Let $M$ and $N$ be connected aspherical Poincar\'e complexes of dimension $n=4m$ with good fundamental groups $G$ and $H$.  If $\Theta\colon \widehat G\to\widehat H$ is an isomorphism, then  $\sigma(M)\equiv \sigma(N)\pmod {8}$.
    Furthermore, if the signatures are equal, then the intersection forms are equivalent over $\QQ$.
\end{thm}
\begin{proof}
    By \cref{profinite_completion_of_cohomology}, for all primes $p$,  the maps     
    \begin{align}
        (\phi_p)_*: H^*(G; \Z) &\to H^*(G; \Z_p)
    \end{align} are induced by pro-$p$ completion. The proof also shows that $(H^*(G; \Z_p))^{\tf}$ is canonically isomorphic to $(H^*(G; \Z))^{\tf} \otimes \Z_p$, so the matrix representing the intersection form on $(H^{2m}(G; \Z))^{\tf}$, when viewed as a matrix over $\Z_p$, is equivalent to the matrix representing the form on $(H^{2m}(G; \Z_p))^{\tf}$, and the rank of $(H^*(G; \Z_p))^{\tf}$ as a (free) $\Z_p$-module is equal to the rank of $(H^*(G; \Z))^{\tf}$ as a (free) $\Z$-module . 
    
    By \cref{profinite_algebra_iso} this coincides with the $\Z_p$-form on $(\mathbf{H}^{2m}(\widehat G; \Z_p))^{\tf}$. Note that this only gives an equivalence of forms over $\Z_p$, not over $\Z$. Since $G$ and $H$ have isomorphic profinite completions, for any $p$, the induced isomorphism between $(\mathbf{H}^{2m}(\widehat G; \Z_p))^{\tf}$ and $(\mathbf{H}^{2m}(\widehat H; \Z_p))^{\tf}$, combined with the earlier observation on ranks, implies that the dimensions of $(H^*(G; \Z))^{\tf}$ and $(H^*(H; \Z))^{\tf}$ are equal.  the induced isomorphism between $(\mathbf{H}^{2m}(\widehat G; \Z_p))^{\tf}$ and $(\mathbf{H}^{2m}(\widehat H; \Z_p))^{\tf}$, with their corresponding intersection pairings, combined with the isomorphism induced by profinite completion shows that $(H^{2m}(G; \Z))^{\tf}$ and $(H^{2m}(H; \Z))^{\tf}$ have the same $p$-excesses when $p>2$ and the same oddities when $p=2$, since excesses and oddities are invariants of the quadratic form. It follows from \cref{sum_formula} that the signatures are congruent modulo $8$. 
    
    We now prove the `furthermore'. Let $A$ and $B$ be matrices representing the bilinear forms on $H^*(G; \Z)$ and $H^*(H; \Z)$ respectively. If the signatures of the intersection forms are equal, then $\det A$ and $\det B$ have the same sign, so the ratio $q\colon= \frac{\det A}{\det B}$ is positive. The forms are equivalent over $\QQ_p$ for every $p$, so there exists some invertible matrix $C_p$ such that $A=C_pBC_p^T$. This implies that, for every $p$, $v_p(q)=v_p(\det C_p \det C_p^T)= 2v_p(\det C_p)$ is even. Hence $q$ is a square in $\QQ^{\times}$. Since the forms have the same oddity and $p$-excesses modulo $8$ by the previous paragraphs, \cref{Hasse} implies that they are equivalent over $\QQ$.
\end{proof}

\begin{remark}[The de Rham and Kervaire invariants]
    One can view the quadratic and symmetric $L$-groups, $L_n(\Z)$ and $L^n(\Z)$, of $\Z$ as analogues of the signature.  The group $L_{4k+2}(\Z)\cong \Z/2$ corresponds to the Arf invariant of a quadratic form and topologically corresponds to the \emph{Kervaire invariant} of a $4k+2$-manifold \cite{Kervaire1960}.  It is known the Kervaire invariant can only be non-zero in dimensions 2, 6, 14, 30, 62, and 126 \cite{MahowaldTangora1967,Browder1969,BarrattJonesMahowald1984,HillHopkinsRavenel2016,LinWangXu2025}. The group $L^{4k+1}(\Z)\cong\Z/2$ detects the \emph{de Rham invariant} of a $4k+1$-manifold \cite{LusztigMilnorPeterson1969,MorganSullivan1974}. The de Rham invariant of $M$ is in fact equal to the Stiefel--Whitney class $w_2w_{4k-1}$ and so is a profinite invariant amongst aspherical manifolds with good fundamental group by \Cref{thmx:Invariance_of_SW_classes}.  
\end{remark}

\subsection{Anthology}\label{sec:detect:thmA}
In this section we prove \Cref{thmx:A}.

\begin{duplicate}[\Cref{thmx:A}]
        Let $M$ and $N$ be smooth closed connected aspherical manifolds with good fundamental groups $G$ and $H$.  If $\widehat G\cong \widehat H$, then
    \begin{enumerate}
        \item $M$ and $N$ are cobordant;
        \item $M$ admits a spin structure if and only $N$ admits a spin structure;
        \item $M$ admits a spin$^\CC$ structure if and only $N$ admits a spin$^\CC$ structure.
        \item $\sigma(M)\equiv \sigma(N)\pmod{8}.$
    \end{enumerate}
\end{duplicate}
\begin{proof}
    (1) follows from \Cref{cobordant}.  (2) follows from \Cref{spin}.  (3) follows from \Cref{spinC}.  (4) follows from \Cref{intersectionForm}.
\end{proof}

\section{A limiting example}\label{sec:example}
The authors are extremely grateful to Holger Kammeyer for providing the following example that illustrates what can go wrong without the goodness hypothesis.

\begin{example}\label{ex:diffDims}
    Take a real quadratic number field $k$ (e.g. $k=\QQ(\sqrt 2)$). Then there exists a simply-connected simple algebraic $k$-group $G_1$ of exceptional type $E_8$ which becomes the real split form $E_{8(8)}$ at one real place of $k$ and which becomes the anisotropic real form at the other real place of $k$. 
    Similarly, there exists a simply-connected simple algebraic $k$-group $G_2$ of type $E_8$ which becomes the real form $E_{8(-24)}$ at one real place of $k$ and the anisotropic real form at the other real place of $k$. 
    The existence is a consequence of the Hasse principle for the Galois cohomology of simply-connected groups.
    In most cases, the Hasse principle was proved through work of Kneser and Harder (see \cite{Kneser1966} for an exposition), but the case of $E_8$, which is the one we need, was proved by Chernousov \cite{Chernousov1989}.
    For $i=1,2$, let $\Gamma_i$ be an arithmetic subgroup of $G_i$. 
    Then by the Borel--Harish-Chandra theorem \cite{Borel1962}, $\Gamma_1$ is a cocompact lattice in the symmetric space of the real Lie group $E_{8(8)}$ and $\Gamma_2$ is a cocompact lattice in the symmetric space of $E_{8(-24)}$.

Now, the simply-connected groups of type $E_8$ have trivial centre which implies that over $p$-adic fields, there exists only one form of type $E_8$ and this splits. 
Lattices of type $E_8$ have the congruence subgroup property \cite{Rapinchuk}, so the groups $\Gamma_1$ and $\Gamma_2$ are profinitely commensurable, and hence suitable finite index subgroups have isomorphic profinite completions. 
For $i=1,2$, virtual cohomological dimension of $\Gamma_i$ is the dimension of the associated symmetric space $X_i$, which is $128$ for $\Gamma_1$ and $112$ for $\Gamma_2$.  Let $\Lambda_i\leqslant \Gamma_i$ for $i=1,2$ be torsion-free finite index subgroups with $\widehat \Lambda_1\cong\widehat \Lambda_2$.  In particular, the closed manifolds $M_i=X_i/\Lambda_i$ for $i=1,2$ have profinitely isomorphic residually finite fundamental groups but $\dim M_1=128 \neq 112=\dim M_2$, so they cannot be cobordant. This does not contradict \cref{thmx:A} because the manifolds in question have fundamental groups which are not good (see \cref{rmk:csp}).
\end{example}

\bibliographystyle{halpha}
\bibliography{refs.bib}

\end{document}

%% file: top.tex
\newcounter{commentcounter}

\usepackage{amsmath} 
\usepackage{amssymb} 
\usepackage{amsthm} 
\usepackage{stmaryrd} 
\usepackage[english]{babel} 
\usepackage[font=small,justification=centering]{caption} 
\usepackage[nodayofweek]{datetime}
\usepackage[shortlabels]{enumitem} 
\usepackage[T1]{fontenc} 
\usepackage[utf8]{inputenc} 
\usepackage{ifthen} 
\usepackage{mathabx} 
\usepackage{mathtools} 
\usepackage[dvipsnames]{xcolor} 
\usepackage[pdftex,  colorlinks=true, pagebackref=true]{hyperref} 
    \hypersetup{urlcolor=RoyalBlue, linkcolor=RoyalBlue,  citecolor=black}
\usepackage{setspace} 
\usepackage{tikz-cd} 
\usetikzlibrary{fit,shapes.geometric}
\usepackage{xfrac} 
\usepackage[capitalize]{cleveref} 

\renewcommand*{\backref}[1]{}
\renewcommand*{\backrefalt}[4]
{
    \ifcase #1
        No citation in the text.
    \or
        Cited on Page #2.
    \else
        Cited on Pages #2.
    \fi
}

\makeatletter
\@namedef{subjclassname@1991}{Mathematical subject classification 1991}
\@namedef{subjclassname@2000}{Mathematical subject classification 2000}
\@namedef{subjclassname@2010}{Mathematical subject classification 2010}
\@namedef{subjclassname@2020}{Mathematical subject classification 2020}
\makeatother

\newtheorem{thm}{Theorem}[section]
\newtheorem{lemma}[thm]{Lemma}
\newtheorem{corollary}[thm]{Corollary}
\newtheorem{prop}[thm]{Proposition}

\newtheorem{thmx}{Theorem}

\theoremstyle{definition}
\newtheorem{defn}[thm]{Definition}
\newtheorem{remark}[thm]{Remark}

\newtheorem{example}[thm]{Example}

\theoremstyle{plain}

    \newtheoremstyle{TheoremNum}
        {8.0pt plus 2.0pt minus 4.0pt}{8.0pt plus 2.0pt minus 4.0pt} 
        {\itshape} 
        {-0.185cm} 
        {\bfseries} 
        {.} 
        { }  
        {\thmname{#1}\thmnote{ \bfseries #3}}
    \theoremstyle{TheoremNum}
    \newtheorem{duplicate}{}

\newcommand*{\claimproofname}{My proof}

\DeclareMathOperator{\Hom}{\mathrm{Hom}}
\DeclareMathOperator{\Ext}{\mathrm{Ext}}
\DeclareMathOperator{\Tor}{\mathrm{Tor}}

\newcommand{\tf}{{\mathrm{tf}}}

\newcommand{\FP}{\mathsf{FP}}

\newcommand{\calh}{{\mathcal{H}}}

\newcommand{\calm}{{\mathcal{M}}}

\newcommand{\Ab}{\mathbf{Ab}}

\DeclareMathOperator{\id}{id}

\newcommand{\Sq}{\mathrm{Sq}}

\newcommand{\RP}{\mathbb{R}\mathbf{P}} 
\newcommand{\CP}{\mathbb{C}\mathbf{P}} 

\DeclareMathOperator{\rank}{\mathrm{rank}}

\DeclareMathOperator{\cd}{\mathrm{cd}}

\def\Z{\mathbb{Z}}

\newcommand{\CC}{\mathbb{C}}

\newcommand{\QQ}{\mathbb{Q}}
\newcommand{\FF}{\mathbb{F}}

\usepackage{tikz}
\usetikzlibrary{arrows,quotes}
\tikzstyle{blackNode}=[fill=black, draw=black, shape=circle]

\tikzcdset{scale cd/.style={every label/.append style={scale=#1},
    cells={nodes={scale=#1}}}}